\documentclass{svproc}
\usepackage{url}
\usepackage{amsmath}
\usepackage{amssymb}
\usepackage[inline]{enumitem}
\usepackage{mathrsfs}

\newcommand{\evid}[1]{\textsf{#1}}
\newcommand{\fin}{\mathrm{fin}}

\begin{document}
\mainmatter              
\title{On the isomorphism problem \\ for power semigroups}
\titlerunning{On the isomorphism problem for power semigroups}  
%
\author{Salvatore Tringali\inst{1}}
\authorrunning{Salvatore Tringali} 
%
\tocauthor{Salvatore Tringali}
\institute{School of Mathematical Sciences, Hebei Normal University \\ Shijiazhuang, Hebei province, 050024 China\\
\email{salvo.tringali@gmail.com},\\ WWW home page:
\texttt{http://imsc.uni-graz.at/tringali}}

\maketitle              

\begin{abstract}
Let $\mathcal P(S)$ be the semigroup obtained by equipping the family of all non-empty subsets of a (multiplicatively written) semigroup $S$ with the operation of setwise multiplication induced by $S$ itself. 
We call a sub\-sem\-i\-group $P$ of $\mathcal P(S)$ downward complete if any element of $S$ lies in at least one set $X \in P$ and any non-empty subset of a set in $P$ is still in $P$.

We obtain, for a commutative semigroup $S$, a characterization of the cancellative elements of a downward complete subsemigroup of $\mathcal P(S)$ in terms of the cancellative elements of $S$. Consequently, we show that, if $H$ and $K$ are cancellative semigroups and either of them is commutative, then every isomorphism from a downward complete subsemigroup of $\mathcal P(H)$ to a downward complete subsemigroup of $\mathcal P(K)$ restricts to an isomorphism from $H$ to $K$. 
This solves a special case of a problem of Tamura and Shafer from the late 1960s and generalizes a recent result by Bienvenu and Geroldinger, where it is assumed, among other conditions, that $H$ and $K$ are numerical monoids.
\keywords{Isomorphism problems, numerical monoids, power algebras, cancellative semigroups, sumsets, setwise products.}
\end{abstract}
\section{Introduction}

Let $S$ be a semigroup (see the end of the section for notation and terminology). Equipped with the binary op\-er\-a\-tion of setwise multiplication induced by $S$ on its own power set and defined by
$$
(X, Y) \mapsto \{xy \colon x \in X,\, y \in Y\},
$$
the \emph{non-empty} subsets of $S$ form a semigroup, herein denoted by $\mathcal P(S)$ and called the \evid{large power semigroup} of $S$.
Furthermore, the non-empty \emph{finite} subsets of $S$ make up a sub\-semi\-group of $\mathcal P(S)$, which we denote by $\mathcal P_\fin(S)$ and call the \evid{finitary power semigroup} of $S$ (see~Example \ref{exa:1.1} for a generalization). Below, we will refer to either of these structures as a \evid{power semigroup} of $S$.

Power semigroups offer a natural algebraic framework for many important problems in ad\-di\-tive combinatorics and related fields, from Ostmann's conjecture \cite[p.~13]{Ostm-1968} on the ``asymptotic additive irreducibility'' of the set $\mathbb P$ of (positive rational) primes to S\'ark\"ozy's conjecture \cite[Conjecture 1.6]{Sar2002} on the ``additive irreducibility'' of the set of quadratic residues modulo $p$ for all large $p \in \mathbb P$. More broadly, the arithmetic of power semigroups --- with a focus on questions related to the possibility or impossibility of factoring certain sets as a product of other sets that are, in a way, ``irreducible'' --- is a rich topic in itself \cite{Fa-Tr18} \cite{An-Tr18} \cite{Bie-Ger-22} \cite{Tri-Yan2023(a)} and has been serving as a guiding example in the development of a ``unifying theory of factorization'' \cite{Tr20(c)} \cite{Co-Tr-21(a)} \cite{Cos-Tri2023} \cite{Cos-Tri2023b} \cite{Tri23} whose primary aim is to extend the classical theory \cite{GeHK06} \cite{GeZh19} beyond its current boundaries and make it usable in non-traditional settings.

As it turns out, $\mathcal P(S)$ was first \emph{systematically} studied by Tamura and Shafer \cite{Tam-Sha1967} \cite{Ham-Nor2009} in the 1960s (its definition can at least be traced back to the early literature on power algebras \cite[Sect.~2]{Bri-1993}) and gained substantial attention in the 1980s and 1990s due to its role in the theory of automata and formal languages \cite{Alm02} \cite{Pin1995}. This led to the following question, arisen from Tamura and Shafer's work and commonly known as the \evid{isomorphism problem for power semigroups}.
\begin{question}
Suppose that a semigroup $H$ is \evid{globally isomorphic} to a semigroup $K$, meaning that $\mathcal P(H)$ is (semigroup-)isomorphic to $\mathcal P(K)$. Is it necessarily true that $H$ is isomorphic to $K$? 
\end{question}

The question was quickly answered in the negative by Mogiljanskaja \cite{Mog1973}, but remains open for \emph{finite} semigroups \cite[p.~5]{Ham-Nor2009} (despite some authors claiming differently based on results that Tamura himself announced, \emph{without proof}, in \cite{Tam-1987}). More generally, one can ask:

\begin{question}\label{quest:2}
Given a class $\mathcal C$ of semigroups, prove or disprove that $H$ is globally isomorphic to $K$, for some $H, K \in \mathcal C$, if and only if $H$ is isomorphic to $K$.
\end{question}

It is easy to check (see Remark \ref{rem:2.1(3)} and cf.~\cite[Remark 1.1]{Tri-Yan2023(a)}) that, if $f \colon H \to K$ is a semigroup iso\-mor\-phism, then the same is true of the function 
$$
\mathcal P(H) \to \mathcal P(K) \colon \allowbreak X \mapsto f[X],
$$
where $f[X] := \{f(x) \colon x \in X\} \subseteq K$ is the (direct) image of $X$ under $f$. Therefore, the interesting aspect of Question \ref{quest:2} lies entirely in the ``only if'' condition.
E.g., the answer to the question is positive for groups \cite{Sha67}, completely $0$-simple semi\-groups \cite{Tam-86} \cite{Tam-90}, and Clifford semi\-groups \cite{Gan-Zhao-2014}; is negative for involution semi\-groups \cite{Crve-Doli-Vinc-2001}; and appears to be unknown for cancellative semigroups (see \url{https://mathoverflow.net/questions/456604/} for further discussion).

The present paper adds to this line of research. More precisely, we say that a subsemigroup $P$ of the large power semigroup of a semigroup $S$ is \evid{downward complete} if each element of $S$ lies in at least one set $X \in P$ and every non-empty subset of a set in $P$ is itself in $P$. The following examples illustrate the property of downward completeness and will come in handy later.

\begin{example}
\label{exa:1.1}
$\mathcal P(S)$ and $\mathcal P_\fin(S)$ are downward complete subsemigroups of $\mathcal P(S)$ for any semigroup $S$; and so is the set of all one-element subsets of $S$, which is obviously isomorphic to $S$ through the embedding $S \to \mathcal P(S) \colon x \mapsto \{x\}$. 

More generally, let $\kappa$ be an \emph{infinite} cardinal number. By basic principles of set theory (in any of the standard foundations of mathematics), each of the families 
$$
\mathcal P_{< \kappa}(S) := \{X \subseteq S \colon 1 \le |X| < \kappa\}
$$
and
$$
\mathcal P_{\le \kappa}(S) := \{X \subseteq S \colon 1 \le |X| \le \kappa\}
$$
is a downward complete subsemigroup of $\mathcal P(S)$. In particular, we have that $\mathcal P_{< \kappa}(S) = \mathcal P_\fin(S) $ when $\kappa$ is the cardinality of $\mathbb N$, and $\mathcal P_{\leq \kappa}(S) =\mathcal P(S)$ when $\kappa$ is the maximum between $|\mathbb N|$ and $|S|$.
\end{example}


\begin{example}
\label{exa:1.2}
Let $S$ be a semigroup and $\sim$ be an equivalence relation on (the carrier set of) $S$ with the additional property that $x_1 \sim y_1$ and $x_2 \sim y_2$ imply $x_1 x_2 \sim y_1 y_2$, i.e., $\sim$ is a (\evid{semigroup}) \evid{congruence} on $S$. Denote by $\mathcal P_\sim(S)$ the family of all non-empty sets $X \subseteq S$ such that $x \sim y$ for any $x, y \in X$.

Every non-empty subset of a set in $\mathcal P_\sim(S)$ is, of course, itself in $\mathcal P_\sim(S)$; and since equivalence relations are reflexive, it is also clear that $\{x\} \in \mathcal P_\sim(S)$ for each $x \in S$. On the other hand, we are guaranteed by the general properties of congruences that $XY \in \mathcal P_\sim(S)$ for all $X, Y \in \mathcal P_\sim(S)$, because $XY \subseteq [xy]_\sim$ for every $x \in X$ and $y \in Y$, where $[z]_\sim$ is the equivalence class of an element $z \in S$ in the quotient set $S/{\sim}$. Therefore, $\mathcal P_\sim(S)$ is a downward complete subsemigroup of $\mathcal P(S)$. 
\end{example}


\begin{example}
\label{exa:1.3}
Given a semigroup $S$, the intersection of any indexed family of downward complete subsemigroups of $\mathcal P(S)$ is itself a downward complete subsemigroup of $\mathcal P(S)$. The proof is straightforward, on account of the fact that a downward complete subsemigroup of $\mathcal P(S)$ contains, a fortiori, every one-element subset of $S$ (further details are left to the reader).
\end{example}

That being said, we obtain, for a \emph{commutative} semigroup $S$, a characterization of the cancellative elements of a downward complete subsemigroup of $\mathcal P(S)$ in terms of the the cancellative elements of $S$ (Proposition \ref{prop:01}); and we use it in Theorem \ref{thm:2.4} to prove that, if $H$ and $K$ are cancellative semi\-groups and either of them is commutative,
then any isomorphism from a downward complete sub\-se\-mi\-group of $\mathcal P(H)$ to a downward complete subsemigroup of $\mathcal P(K)$ maps singletons to singletons and hence restricts to an iso\-mor\-phism from $H$ to $K$ (up to the embedding mentioned in Example \ref{exa:1.1}).
It follows (Corollary \ref{cor:2.4}) that $H$ and $K$ are globally iso\-mor\-phic if and only if they are isomorphic, if and only if $\mathcal P_\fin(H)$ is isomorphic to $\mathcal P_\fin(K)$. 
We can thus give a positive answer to Question \ref{quest:2} for the class of cancellative commutative se\-mi\-groups and extend a recent result of Bienvenu and Geroldinger \cite[Theorem 3.2(3)]{Bie-Ger-22}, where $H$ and $K$ are numerical monoids (i.e., submonoids of the additive monoid of non-negative integers with finite complement in $\mathbb N$) with $\mathcal P_\fin(H)$ isomorphic to $\mathcal P_\fin(K)$ (see the comments at the end of Sect.~\ref{sect:2}).

\subsection*{Generalities} 
We denote by $\mathbb N$ the (set of) non-negative integers, by $\mathbb N^+$ the positive integers, and by $|\cdot|$ the cardinality of a set. Unless otherwise specified, we write all semigroups multiplicatively. We refer to Howie's monograph \cite{Ho95} for the basic theory of semigroups and monoids. 
In particular, we recall that, given a semigroup $S$, an element $a \in S$ is \evid{left} (resp., \evid{right}) \evid{cancellative} if the map $S \to S \colon x \mapsto ax$ (resp., $S \to S \colon x \mapsto xa$) is injective; and is \evid{cancellative} if it is left and right cancellative.
The semigroup itself is then said to be cancellative if each of its elements is cancellative. Further notation and terminology, if not explained upon first use, are standard or should be clear from the context. 

\section{Results}
\label{sect:2}

The main result of the paper (namely, Theorem \ref{thm:2.4}) is actually a direct consequence of Proposition \ref{prop:01} below, which may be of independent interest. Let us start with a basic remark.

\begin{remark}\label{rem:singletons}
Let $u$ be a left (resp., right) cancellative element in a semigroup $S$, and assume $uX = uY$ (resp., $Xu = Yu$) for some $X, Y \in \allowbreak \mathcal P(S)$. Given $x \in X$, there then exists $y \in Y$ such that $ux = uy$ (resp., $xu = \allowbreak yu$). It follows (by the hypothesis on $u$) that $x = \allowbreak y$ and hence $X \subseteq \allowbreak Y$. By symmetry, we can thus conclude that $X = Y$, which ultimately shows that $\{u\}$ is a left (resp., right) cancellative element in $\mathcal P(S)$.
\end{remark}

The key idea, simple as it may be, is to establish that, in a downward complete subsemigroup of the large power semigroup of a cancellative \emph{commutative} semigroup $S$, the cancellative elements are all and only the one-element subsets of $S$. In fact, we will prove something slightly more general.

\begin{proposition}
\label{prop:01}
Let $S$ be a commutative semigroup. The cancellative elements of a downward complete subsemigroup of $\mathcal P(S)$ are all and only the singletons $\{u\}$ such that $u$ is a cancellative element in $S$.
\end{proposition}

\begin{proof}
Let $P$ be a downward complete subsemigroup of $\mathcal P(S)$. If $ux = uy$ for some $u, x, y \in S$ with $x \ne \allowbreak y$, then $\{u\} \allowbreak \{x\} = \{u\} \{y\}$ in $\mathcal P(S)$. Since all one-element subsets of $S$ are in $P$, it is thus clear from Remark \ref{rem:singletons} that $\{u\}$ is a cancellative element in $P$ if and only if $u$ is cancellative in $S$.

It remains to check that every cancellative element of $P$ is a singleton. For, fix $A \in \allowbreak P$ with $|A| \ge \allowbreak 2$. We have to prove that the function 
$$
f \colon P \to P \colon X \mapsto AX
$$
is non-injective. We distinguish two cases.

\vskip 0.05cm

\textsc{Case 1:} There exist $a, b \in A$ with $a^2 = ab$ and $a \ne b$. Set $B := A \setminus \{a\}$. If $a \ne x \in A$, then $x \in B$ and hence $xa = ax \in AB$ (by the commutativity of $S$). Moreover, $a^2 = ab \in AB$ (since $a \ne b$ yields $b \in B$). It follows that $Aa \subseteq AB$ and hence $A^2 = AB \cup Aa \subseteq AB \subseteq \allowbreak A^2$, namely, $f(A) = f(B)$. But this means that $f$ is not an injection (as wished), because $P$ is a downward complete sub\-semi\-group of $\mathcal P(S)$ and $B$ is a non-empty proper subset of $A$ (with the result that $A, B \in P$ and $A \ne B$).

\vskip 0.05cm

\textsc{Case 2:} We are not in \textsc{Case 1}. Pick $a, b \in A$ with $a \ne b$ (recall that $|A| \ge 2$) and set $B := A^2 \setminus \allowbreak \{ab\}$. Since $b^2 \ne ab$ (or else we would fall back to \textsc{Case 1}), we have $b^2 \in B$ and hence $\emptyset \ne B \subsetneq A^2$. Given that $P$ is a downward complete subsemigroup of $\mathcal P(S)$, it follows that $A^2, B \in P$ and $A^2 \ne B$.

Now, let $x \in A$. If $xa = ab$, then it is obvious that $xab = ab^2 \in AB$ (by the fact that $ab \ne \allowbreak b^2 \in \allowbreak B$). Otherwise, $xa \in B$ and again $xab = bxa \in AB$ (by the commutativity of $S$). Consequently, we find that $Aab \subseteq AB$ and hence $A^3 = AB \cup Aab \subseteq AB \subseteq A^3$, namely, $f(A^2) = f(B)$. So, $A$ is a non-cancellative element of $P$ (as we have already observed that $A^2, B \in P$ and $A^2 \ne B$).
\end{proof}

The next example demonstrates that Proposition \ref{prop:01} does not hold for non-commutative cancellative semigroups. This suggests that addressing Question \ref{quest:2} in the cancellative, non-commutative case (assuming the answer is still affirmative) may necessitate a somewhat different approach than the one envisaged in this work for the commutative case.

\begin{example}
Let $V$ be a non-empty set and $\mathscr F^+(V)$ be the free semigroup over $V$. The underlying set of $\mathscr F^+(V)$ consists of all the non-empty finite tuples with components in $V$; and the semigroup operation, hereby denoted by the symbol $\ast\,$, is the (ordered) concatenation of such tuples. 

As a matter of fact, $\mathscr F^+(V)$ is a cancellative semigroup. We claim that every non-empty subset $X$ of $V$ is a cancellative element in the large power semigroup of $\mathscr F^+(V)$, whose operation we continue to denote by $\ast\,$. In particular, this will prove (taking $|V| \ge 2$) that a cancellative element of the large power semigroup of a cancellative semigroup need not be a singleton (in contrast with Proposition \ref{prop:01}).

For the claim, assume $X \ast Y_1 = X \ast Y_2$ for some non-empty sets $Y_1, Y_2 \subseteq \mathscr F^+(V)$ (the case when $Y_1 \ast \allowbreak X = \allowbreak Y_2 \ast X$ is symmetric). We need to check that $Y_1 = Y_2$. To this end, it is clear that, since the elements of $\mathscr F^+(V)$ can be \emph{uniquely} represented as a (non-empty, finite, ordered) concatenation of elements of $V$, the sets $x_1 \ast Y_1$ and $x_2 \ast Y_2$ are disjoint for all $x_1, x_2 \in X$ with $x_1 \ne x_2$. It follows that $X \ast Y_1 = X \ast Y_2$ if and only if $x \ast Y_1 = x \ast Y_2$ for each $x \in \allowbreak X$, which, by Remark \ref{rem:singletons} and the non-emptiness of $X$, implies $Y_1 = Y_2$ (as wished).
\end{example}

We continue with a few additional observations that will prove useful shortly afterwards.

\begin{remark}
\label{rem:2.1(1)}
Let $f \colon H \to K$ be a semigroup isomorphism and $x$ be a left cancellative element of $H$, and pick $u, v \in K$. Since $f$ is onto, there exist $y, z \in H$ such that $u = f(y)$ and $v = f(z)$. Consequently, $f(x) u = f(x) v$ if and only if $f(xy) = f(xz)$, which implies, by the injectivity of $f$, that $xy = xz$. It is thus obvious from the left cancellativity of $x$ in $H$ that $y = z$ and hence $u = f(y) = f(z) = v$.

By the left-right symmetry inherent in the previous argument, we can therefore conclude that $f$ maps left (resp., right) cancellative elements of $H$ to left (resp., right) cancellative elements of $K$. In particular, it maps cancellative elements to cancellative elements. 
\end{remark}


%


\begin{remark}
\label{rem:2.1(3)}
Let $H$ and $K$ be semigroups, and let $f$ be an isomorphism from a downward complete subsemigroup $P$ of $\mathcal P(H)$ to a downward complete sub\-sem\-i\-group $Q$ of $\mathcal P(K)$. In addition, assume $H$ is commutative. Given $u, v \in K$, there then exist $X, Y \in P$ with $f(X) = \{u\}$ and $f(Y) = \{v\}$, since $f$ is surjective and $Q$ contains every one-element subset of $K$. On the other hand, it is straightforward from the commutativity of $H$ that $XY = YX$.
As a result, we have $\{u\} \{v\} = f(XY) = f(YX) = \{v\} \{u\}$ and hence $uv = vu$, which ultimately proves that $K$ is itself a commutative semigroup (cf.~Proposition 1.3(a) in \cite{Pin1987}).
\end{remark}

\begin{remark}
\label{rem:2.1(4)}
Assume $f \colon H \to K$ is a semigroup homomorphism and let $F$ be the function 
$
\mathcal P(H) \to \mathcal P(K) \colon X \mapsto f[X]$. It is immediate that
\begin{equation*}
\begin{split}
F(XY) & = \{f(xy) \colon x \in X,\, y \in Y\} = \{f(x)f(y) \colon x \in X,\, y \in Y\} = F(X) F(Y),
\end{split}
\end{equation*}
for all $X, Y \in \mathcal P(H)$. Consequently, $F$ is a  homomorphism $\mathcal P(H) \to \mathcal P(K)$.

On the other hand, if $f$ is injective and $F(X_1) = F(X_2)$ for some $X_1, X_2 \in \mathcal P(H)$, then $X_1 = X_2$, and hence $F$ is itself injective; otherwise, there would exist either $x \in X_1$ with $f(x) \in F(X_1) \setminus F(X_2)$ or $x \in X_2$ with $f(x) \in F(X_2) \setminus F(X_1)$. 
Moreover, if $f$ is surjective and $Y$ is a non-empty subset of $K$, then $X := \{x \in X \colon f(x) \in Y\}$ is a non-empty subset of $H$ with $F(X) = Y$, which ultimately means that also $F$ is surjective.

Now, suppose that $f$ is an isomorphism. It then follows from the above that $F$ is an iso\-mor\-phism from $\mathcal P(H)$ to $\mathcal P(K)$. In addition, $f$ being a bijection implies that a subset of $H$ and its image under $f$ have the same cardinality. 

It is therefore clear (see Example \ref{exa:1.1} for the notation) that, for every infinite cardinal $\kappa$, the function $F$ restricts to an isomorphism $\mathcal P_{<\kappa}(H) \to \mathcal P_{<\kappa}(K)$ as well as to an isomorphism $\mathcal P_{\le \kappa}(H) \to \mathcal P_{\le \kappa}(K)$. Most notably, $F$ restricts to an iso\-mor\-phism from $\mathcal P_\fin(H)$ to $\mathcal P_\fin(K)$.
\end{remark}

We are finally ready to prove our main result and show, in the subsequent corollary, that Question \ref{quest:2} has a positive answer in the class of cancellative commutative semigroups.

\begin{theorem}\label{thm:2.4}
If $H$ and $K$ are cancellative semigroups and either of them is commutative, then every isomorphism from a downward complete subsemigroup of $\mathcal P(H)$ to a downward complete subsemigroup of $\mathcal P(K)$ restricts to an isomorphism $H \to K$.
\end{theorem}

\begin{proof}
Let $P$ be a downward complete subsemigroup of $\mathcal P(H)$ and $Q$ be a downward complete subsemigroup of $\mathcal P(K)$, and suppose that there is an isomorphism $f$ from $P$ to $Q$. Since either of $H$ and $K$ is commutative (by hypothesis), we have from Remark \ref{rem:2.1(3)} that both $P$ and $Q$ are. On the other hand, each of $H$ and $K$ is cancellative. Therefore, we gather from Proposition \ref{prop:01} that the cancellative elements of $\mathcal P(H)$ (resp., of $\mathcal P(K)$) are all and only the one-element subsets of $H$ (resp., of $K$). It then follows by Remark \ref{rem:2.1(1)} that, for every $x \in H$, there is a uniquely determined $y \in K$ such that $f(\{x\}) = \allowbreak \{y\}$. Considering that the same argument also applies to the functional inverse of $f$, this suffices to conclude that $f$ restricts to an iso\-mor\-phism from $H$ to $K$ (as wished).
\end{proof}

\begin{corollary}\label{cor:2.4} 
If $H$ and $K$ are cancellative semigroups and either of them is commutative, then the following conditions are equivalent:
\begin{enumerate}[label=\textup{(\alph{*})}]
\item\label{cor:2.4(a)} $H$ and $K$ are globally isomorphic, i.e., $\mathcal P(H)$ is isomorphic to $\mathcal P(K)$.
\item\label{cor:2.4(b)} $\mathcal P_\fin(H)$ is isomorphic to $\mathcal P_\fin(K)$.
\item\label{cor:2.4(c)} $H$ is isomorphic to $K$.
\end{enumerate}
\end{corollary}

\begin{proof}
We noted in Example \ref{exa:1.1} that $\mathcal P(S)$ and $\mathcal P_\fin(S)$ are both downward complete sub\-semi\-groups of $\mathcal P(S)$ for every semigroup $S$. So, it is immediate from Theorem \ref{thm:2.4} that \ref{cor:2.4(a)} $\Rightarrow$ \ref{cor:2.4(c)} and \ref{cor:2.4(b)} $\Rightarrow$ \ref{cor:2.4(c)}. This concludes the proof, as we know from Remark \ref{rem:2.1(4)} that \ref{cor:2.4(c)} $\Rightarrow$ \ref{cor:2.4(a)} and \ref{cor:2.4(c)} $\Rightarrow$ \ref{cor:2.4(b)}. 
\end{proof}

Corollary \ref{cor:2.4} is a generalization of a theorem of Bienvenu and Geroldinger \cite[Theorem 3.2(3)]{Bie-Ger-22} ac\-cord\-ing to which the finitary power semigroups of two numerical monoids $H$ and $K$ are isomorphic if and only if $H = \allowbreak K$. But there is a subtle detail here to work out.

In fact, Bienvenu and Geroldinger's result is actually about \emph{monoid} iso\-mor\-phisms (note that the finitary power semigroup of a monoid $S$ with identity $1_S$ is a monoid with identity $\{1_S\}$). 
However, two numerical monoids are isomorphic if and only if they are equal \cite[Theorem 3]{Hig-969}; and it is easy to show that a surjective \emph{semigroup} homomorphism $f$ from a monoid $H$ to a monoid $K$ is, a fortiori, a \emph{monoid} homomorphism, i.e., it maps the identity $1_H$ of $H$ to the identity $1_K$ of $K$. 

For, set $e := f(1_H)$. Each $y \in \allowbreak K$ is the image under $f$ of an element $x \in \allowbreak H$, which implies $ye = \allowbreak f(x)  f(1_H) = \allowbreak f(x1_H) = \allowbreak f(x) = \allowbreak y$. It follows, by taking $y = \allowbreak 1_K$, that $f(1_H) = \allowbreak e = \allowbreak 1_K e = \allowbreak 1_K$.

\section*{Acknowledgements}

I am thankful to Paolo Leonetti (Insubria University, Italy) for the idea underlying Example \ref{exa:1.2}, and to Daniel Smertnig (University of Ljubljana, Slovenia) for pointing out that two monoids are monoid-isomorphic if and only if they are sem\-i\-group-isomorphic (cf.~the last lines of Sect.~\ref{sect:2}).

%
%

\end{document}